\documentclass[reqno,11pt]{amsart}
\usepackage{amsmath,amsthm,amssymb,amsfonts,color,pifont,cite}
\usepackage[normalem]{ulem}
\usepackage{framed}
\usepackage[colorlinks=true,urlcolor=blue,citecolor=red,linkcolor=blue,linktocpage,pdfpagelabels,bookmarksnumbered,bookmarksopen]{hyperref}
\usepackage{pgf,tikz}
\usepackage{mathrsfs}
\usetikzlibrary{arrows}

\newtheorem{theorem}{Theorem}[section]

\newtheorem{lemma}[theorem]{Lemma}

\newtheorem{proposition}[theorem]{Proposition}
\theoremstyle{definition}
\newtheorem{remark}[theorem]{Remark}

\def\F{\mathcal{F}}

\def\Cap{{\rm Cap}}

\def\R{\mathbb R}

\def\H{\mathcal{H}}

\def\LM#1{\hbox{\vrule width.2pt \vbox to#1pt{\vfill \hrule width#1pt
height.2pt}}}
\def\LL{{\mathchoice {\>\LM7\>}{\>\LM7\>}{\,\LM5\,}{\,\LM{3.35}\,}}}
\def\restr{{\LL}}

\def\cE{\mathcal{E}}
\def\I{\mathcal{I}}

\def\les{\lesssim}
\def\ges{\gtrsim}
\def\eps{\varepsilon}

\def\lt{\left}
\def\rt{\right}

\def\N{\mathbb{N}}
\renewcommand{\phi}{\varphi}

\author{Michael Goldman}
\address{LJLL, Universit\'e Paris Diderot, CNRS, UMR 7598, Paris, France}
\email{goldman@math.univ-paris-diderot.fr}

\author{Berardo Ruffini}
\address{Institut Montpelli\'erain Alexander Grothendieck, University of Montpellier, 34095 Montpellier Cedex 5, France}
\email{berardo.ruffini@umontpellier.fr}

\numberwithin{equation}{section}

\title[Review on equilibrium shapes]{Equilibrium shapes of charged droplets and related problems: (mostly) a review}

\begin{document}

\begin{abstract}

We review some recent results on the equilibrium shapes of charged liquid drops. We show that the natural variational model is ill-posed and how this can be overcome by either restricting
the class of competitors or by adding penalizations in the functional.  The original contribution of this note is twofold. First, we prove existence of an optimal distribution of charge for a
conducting drop subject to an external electric field. Second, we prove that there exists no optimal conducting drop in this setting.  
\end{abstract}

\maketitle


The main purpose of the   paper  is to  review some recent progress in the study of variational problems describing the shape of conducting
liquid drops. The salient feature of these models is the competition between an interfacial term  with a non-local and repulsive term of capacitary type.
 The somewhat surprising and puzzling fact is that contrarily to the experimental observations these models are generally ill-posed.
 However,  taking into account various possible regularizing mechanisms it 
is possible in some cases to recover well-posedness together with stability results for the ball in the regime of small charges. \\
Alongside this review, we also provide new results on the closely related problem of equilibrium shapes of conducting drops subject to an external electric field. 
We prove that for every fixed drop, the optimal distribution of charges exists but that as for the charged drops model, no equilibrium shape exists. Moreover, we show that well-posedness cannot be recover easily since minimizer do not exist even in the rigid class of convex sets.\\ 

The paper is organized as follows. In Section \ref{section1} we recall the definition of the  capacity and study existence and characterization of optimal distributions of charges.  
In Section \ref{section2} we review some recent results on the charged liquid drop model. In particular, we show ill-posedness of this problem and discuss various  possible regularizations.
In Section \ref{section3} we study  the problem of equilibrium shapes for perfectly  conducting  drops in an external electric field. 
In the last section, we state several open problems.

\section{Equilibrium measures, potentials and capacities}\label{section1}
We start by investigating the optimal distribution of charges for a given compact set  $\Omega\subset \R^N$. Most of the results can be found in \cite{landkof,GNRI,GNRII}.  
For fixed $\alpha\in (0,N)$, and Radon measures $\mu$, $\nu$ we define 
\[
 I_\alpha(\mu,\nu):= \int_{\R^N\times \R^N} \frac{d\mu(x)\,d\nu(y)}{|x-y|^{N-\alpha}}
\]
and let then $I_\alpha(\mu):=I_\alpha(\mu,\mu)$. With a slight abuse of notation, we  also let
\[
 I_N(\mu,\nu):= \int_{\R^N\times \R^N}- \log |x-y| \,d\mu(x)\,d\nu(y)
\]
The  equilibrium measure of a set $\Omega$ is the solution of the problem
\[
\I_\alpha(\Omega):=\min_{\mu(\Omega)=1}I_\alpha(\mu),
\]  
where the class of minimization runs over all probability measures on $\Omega$. We denote by  $\I_\alpha(\Omega)$ the Riesz potential energy of $\Omega$. We can then define the $\alpha-$capacity of a set $\Omega$ as 
\[
 \Cap_\alpha(\Omega):=\frac{1}{\I_\alpha(\Omega)}\quad  \textrm{ if } \alpha<N \qquad \textrm{and} \qquad \Cap_N(\Omega):= e^{-\I_N(\Omega)}.  
\]
For a given measure $\mu$, it is useful to  define the associated  potential
\[
v(x):=\int_{\Omega}\frac{d\mu(y)}{|x-y|^{N-\alpha}},
\]
and its natural counterpart in the logarithmic case $\alpha=N$. We list below some known facts about equilibrium measures and potentials.
\begin{theorem}\label{th:equmeasure}
	Let $\Omega$ be a compact subset of $\R^N$ and $\alpha\in (0,N]$ be such that $\I_\alpha(\Omega)<+\infty$\footnote{this is for instance the case if $\Omega$ is a compact set with $|\Omega|>0$.}. Then, 
		\begin{enumerate}
		\item There exists a unique equilibrium measure $\mu$ \cite[p. 131--133]{landkof}.
	      \item There exists a constant $c(N,\alpha)>0$ such that the potential $v$ satisfies 
		\[(-\Delta)^{\frac{\alpha}{2}} v= c(N,\alpha) \mu\]
		in the distributional sense (here $(-\Delta)^s$ refers to the $s$-Laplacian). 
		In particular $v$ is $\alpha/2-$harmonic outside of the support of $\mu$. Moreover $v=\I_\alpha(\Omega)$ $\alpha-$q.e. on the support of $\mu$\footnote{we say that a property holds $\alpha-$q.e. if it holds up to a set of zero capacity.} 
		and $v\ge \I_\alpha(\Omega)$ $\alpha-$q.e. on $\Omega$ \cite[Lem. 2.11]{GNRI}. 
	      \item For $\alpha<N$, 
		\[
		\lim_{|x|\to\infty}v(x)|x|^{N-\alpha}=1
		\]
		 while for $\alpha=N$,
		\[
		\lim_{|x|\to\infty} v(x) +\log|x|=0.
		\]
		\item For $1<\alpha<2$,  the support of $\mu$ coincides with $\Omega$ \cite[Lem. 2.15]{GNRI}. 
		\item For $\alpha\ge 2$   the support of $\mu$ is contained in $\partial\Omega$ and $v=\I_\alpha(\Omega)$ on $\Omega$ \cite[Lem. 2.15]{GNRI}.
		
	\end{enumerate}
	
\end{theorem}

\begin{remark}
In the claim of point $(5)$ of the previous theorem  the support of $\mu$ is not equal to $\partial \Omega$ in general. For example, if $\Omega$ is the circular annulus
\[
\Omega=\left\{x\in\R^N:\,\frac12\le |x|\le 1 \right\},
\] 
then its optimal measure $\mu$ for $\alpha=2$ is equal to the equilibrium measure of  the ball $B_1$ that is 
\[
\mu=\frac{1}{\H^{N-1}(\partial B_1)}\,\H^{N-1}\restr \partial B_1.
\] 
\end{remark}
\begin{remark}
	The characterization at infinity of the behavior of the potential described in point $(3)$ of the previous theorem has been successfully exploited to show geometric
	inequalities, such as Brunn-Minkowski-type inequalities \cite{caffarelli,NovRuf,borell,ColSal}.  
\end{remark}

\begin{remark}
	For $N\ge3$ and $\alpha>1$ one can show that the (fractional) capacity of a set $\Omega$ can be  characterized as 
	\[
	\Cap_\alpha(\Omega)=\inf\left\{\|(-\Delta)^\frac{\alpha}{2} u\|^2_{L^2(\R^N)}\,:\,u\in C^1_c(\R^N),\,\,u\ge\chi_\Omega \right\}.
	\]
	Notice that for $\alpha=2$,  $\|(-\Delta)^\frac{\alpha}{2} u\|^2_{L^2(\R^N)}$ reduces to the Dirichlet energy of $u$.
\end{remark}
\begin{remark}
 As pointed out in \cite{GNRII}, in the Coulombic case $\alpha=2$, the equilibrium measure $\mu$ coincides with the so-called harmonic measure at infinity of $\Omega$.
\end{remark}

The next result gives a  link between sets of Hausdorff dimension at least $N-\alpha$ and sets of positive capacity (see \cite[Th. 3.13]{landkof}).
\begin{proposition}\label{capdim}
	Let $\Omega\subset\R^N$. Then if the Hausdorff dimension of $\Omega$ is greater than $N-\alpha$, $\Cap_\alpha(\Omega)>0$.
\end{proposition}

We now study the existence of an equilibrium measure for a conducting set subject to an external electric field $\mathbb{E}=-\nabla \phi$. 
We focus on the Coulombic case $\alpha=2$ and $N\ge 3$. For $\mu$ a (signed) measure supported on $\Omega$ with $\mu(\Omega)=0$, the electrostatic energy of $\mu$ is given by 
\[
 F(\mu):=I_2(\mu)+\int_{\Omega} \phi d\mu,
\]
while the electrostatic energy of $\Omega$ is 
\begin{equation}\label{probF}
 \F(\Omega):=\min_{\mu(\Omega)=0} F(\mu).
\end{equation}
Let us notice that in contrast with the situation of Theorem \ref{th:equmeasure}, the existence of minimizers for \eqref{probF} is not straightforward. Indeed, a bound on $I_2(\mu)$ does not give a bound on the total variation of $\mu$ (see for instance \cite[Chap. VI]{landkof}).
This prevents us from using the Direct Method of Calculus of Variations. Instead  we will directly look for a solution of the Euler-Lagrange equation and  prove that it is a minimizer of \eqref{probF}.\\
We start by recalling that \cite[Th. 1.15]{landkof}
\begin{proposition}\label{prop:positiv}
 For every signed measure $\mu$, $I_\alpha(\mu)\ge 0$ and $I_\alpha(\mu)=0$ if and only if $\mu=0$.
\end{proposition}
The next proposition shows that solutions of the Euler-Lagrange equation are minimizers.
\begin{proposition}\label{prop:Eulerimplmin}
 Let $\Omega$ be a compact set. Assume that there exists  a constant $\lambda\in \R$ and  function $v$ solving
 \begin{equation}\label{eq:lambda}
\begin{cases}
-\Delta v=0 &\text{in $\R^N\setminus\Omega$}\\
v=-\frac{\varphi}{2}+\lambda&\text{in $\Omega$}\\
\lim_{|x|\to+\infty}v(x)=0,
\end{cases}
\end{equation}
and such that $\mu:= c(N,2)^{-1} (-\Delta v)$ (where $c(d,2)$ is the constant defined in Theorem \ref{th:equmeasure}) is a measure satisfying $\mu(\Omega)=0$. Then, $\mu$ is the unique minimizer of \eqref{probF}. 
\end{proposition}
\begin{proof}
Assume that $\lambda$ and $v$ exists and let $\mu:=c(N,2)^{-1} (-\Delta v)$. To show that $\mu$ is optimal, we notice that if $\nu$ is another measure with $\nu(\Omega)=0$, then
\[
\begin{aligned}
F(\nu) &=I_2(\nu)+\int_{\Omega}\varphi \,d\nu\\
&=F(\mu)+I_2(\nu-\mu)+2\int_{\Omega\times \Omega}\frac{d(\nu-\mu)(x)\,d\mu(y)}{|x-y|^{N-2}}+\int_{\Omega}\varphi(x)\,d(\nu-\mu)(x)\\
&=F(\mu)+I_2(\nu-\mu)+\int_{\Omega}(2v+\varphi)\,d(\nu-\mu).
\end{aligned} 
\]	
Integrating the equation $2v+\varphi=\lambda$ against $\nu-\mu$ in $\Omega$ and using that $(\nu-\mu)(\Omega)=0$, we get that
\[
F(\nu)=F(\mu)+I_2(\nu-\mu).
\]
Using Proposition \ref{prop:positiv} we conclude the proof.\end{proof}

We are thus left with the construction of a constant $\lambda$ and a function $v$ satisfying the hypothesis of Proposition \ref{prop:Eulerimplmin}. In order to avoid technicalities, we will assume that $\partial \Omega$ and $\phi$ are smooth.
For later use, let $v_\Omega$ be the solution of 
\[
\begin{cases}
-\Delta v_\Omega=0 &\text{in $\R^N\setminus\Omega$}\\
v_\Omega=1&\text{in $\Omega$}\\
\lim_{|x|\to+\infty}v_\Omega(x)=0,
\end{cases}
\]
which can be readily constructed from the measure $\mu_\Omega$ which minimizes $\I_2(\Omega)$.
\begin{proposition}\label{prop:existv}
 For $N\ge 3$, let $\Omega\subset \R^N$ be a smooth compact set and let $f$ be a smooth function on $\Omega$. Then, there exists a unique solution of 
 \begin{equation}\label{eq:f}
\begin{cases}
-\Delta v=0 &\text{in $\R^N\setminus\Omega$}\\
v=f&\text{in $\Omega$}\\
\lim_{|x|\to+\infty}v(x)=0.
\end{cases}
\end{equation}
Moreover, $-\Delta v$ is a bounded measure on $\Omega$. 
\end{proposition}
\begin{proof}
 The uniqueness part of the statement follows by maximum principle. Let us turn to the existence. For $R\gg1$, let $v_R$ be the unique solution of 
 \[
\begin{cases}
-\Delta v_R=0 &\text{in $B_R\setminus\Omega$}\\
v_R=f&\text{in $\Omega$}\\
v_R=0 & \text{on $\partial B_R$}.
\end{cases}
\]
Let $f_+:=\max(\max_{\partial \Omega} f,0)$ and $f_-:=\min(\min_{\partial \Omega} f,0)$ and let $v_\pm:= f_\pm v_\Omega$. By maximum principle, $v_-\le v_R\le v_+$. Since $v_\pm \to 0$ at infinity, using elliptic regularity and letting
$R\to+\infty$ we obtain a function $v$ which satisfies \eqref{eq:f}.\\
To see  that $-\Delta v$ is a measure it is enough to notice that for  $\psi\in C^\infty_c(\R^N)$, we have 
\begin{align*}
 \int_{\R^N} -\Delta \psi v&=\int_{\R^N\backslash\Omega} -\Delta \psi v+\int_{\Omega} -\Delta \psi v\\
 &= \int_{\partial \Omega} \psi \lt[ \lt(\frac{\partial v}{\partial \nu}\rt)^+-\frac{\partial f}{\partial \nu}\rt] +\int_\Omega -\Delta f \psi,
\end{align*}
where $\nu$ is the outward normal to $\partial \Omega$ and where  $\lt(\frac{\partial v}{\partial \nu}\rt)^+$ is the exterior trace of $\frac{\partial v}{\partial \nu}$. 
\end{proof}
We can now prove our main result of this section.
\begin{theorem}
 For every smooth compact set $\Omega$ and every smooth function $\phi$ there exists a unique solution $\mu$ to \eqref{probF}.
\end{theorem}
\begin{proof}
 Let $v_0$ be the solution of \eqref{eq:f} with $f=-\frac{\phi}{2}$ and for $\lambda\in \R$ let $v_\lambda:= v_0+\lambda v_\Omega$. Since $-\Delta v_0$ is a finite measure and $-\Delta v_\Omega$ is a multiple of $\mu_\Omega$ (which is a probability measure),
we can find $\lambda$ such that $v_\lambda$ solves \eqref{eq:lambda} and $-\Delta v_\lambda (\Omega)=0$. This concludes the proof thanks to Proposition \ref{prop:Eulerimplmin}.  

\end{proof}

\section{Equilibrium shapes of charged liquid drops}\label{section2}
The study of  the equilibrium shapes of charged liquid drops started with the seminal paper of Lord Rayleigh \cite{Ray} who calculated through a linear stability analysis 
the maximal charge that a spherical drop can bear before the onset of instability. It was later observed by Zeleny \cite{Zel} that for larger charges, conical singularities  (the so-called Taylor cones) appear  together with the formation of a thin steady jet.
Since then it has been understood that the micro-drops forming the jet carry a large portion of the charge but only a small portion of the mass (see \cite{fernandezdelamora07}).
Because of its numerous applications in particular in mass spectrometry  \cite{gaskell}, this phenomenon has attracted a wide interest in the last thirty years. We refer to \cite{MurNov} for a more detailed discussion on the physical background and literature.
Let us point out that mathematically, very little is known about what happens after the onset of singularities. In particular the formation of the Taylor cones is still badly understood (see \cite{FusJul,Garcia} for some results in this direction).\\
The variational model describing the equilibrium shape of a charged liquid drop is the following. For a given charge $Q>0$ the energy of a compact set $\Omega$ is equal to 
\[
 \cE_\alpha(\Omega):= \H^{N-1}(\partial \Omega)+ Q^2 \I_\alpha(\Omega),
\]
where $\H^{N-1}$ refers to the $(N-1)-$dimensional Hausdorff measure. 
Up to a renormalization of the volume, we are   looking for a solution of 
\begin{equation}\label{eq:mainprob}
 \min_{|\Omega|=|B_1|} \cE_\alpha(\Omega).
\end{equation}
The physical case corresponds to $N=3$, with Coulombic interaction $\alpha=2$. When the charge distribution $\mu$ is taken to be uniform on $\Omega$, this problem is often called the sharp interface Ohta-Kawasaki model (or Gamow's liquid drop model). See \cite{ChMuTop} for a recent overview of this related problem.\\
The mathematical interest of \eqref{eq:mainprob} lies in the fact that there is a competition between the perimeter which is a local term minimized by the ball and the non-local electrostatic energy $\I_\alpha$ which is maximized by the ball (at least for $\alpha\le 2$ \cite{Betsakos}). Based on the experimental observations, the linear stability analysis (see in particular \cite{Ray,FonFri}) and by analogy with what is known for the Ohta-Kawasaki model,
one could expect that for small $Q$ the ball minimizes \eqref{eq:mainprob} while for large $Q$ no global minimizers exist. Surprisingly enough, this is not the case and the problem is always ill-posed when $\alpha\in (1,N)$. Roughly speaking this is due to the fact  
that the perimeter term sees objects of dimension $N-1$ while $\I_\alpha$ naturally lives on object of dimension $N-\alpha$, see Proposition \ref{capdim}. The following non-existence result has been obtained in \cite[Th. 3.2]{GNRI}. 
\begin{theorem}\label{theo:nonexist}
 For every $N\ge 2$, $\alpha\in(1,N)$ and $Q>0$,
 \[
  \inf_{|\Omega|=|B_1|} \cE_\alpha(\Omega)= \H^{N-1}(\partial B_1).
 \]
By the isoperimetric inequality, this means that \eqref{eq:mainprob} is not attained.
\end{theorem}
\begin{proof}
 For $n\in \N$ and $\beta\in ((N-1)^{-1}, (N-\alpha)^{-1})$, let $r_n:=n^{-\beta}$. Consider the competitor $\Omega_n$  made of $n$ balls of radius
 $r_n$  each carrying a charge $n^{-1}$ and infinitely far apart together with a ball of radius $R_n\sim 1$ which is free of charge. We can then compute the energy
 \[
  \cE_\alpha(\Omega_n)= \H^{N-1}(\partial B_{R_n}) +n r_n^{N-1} \H^{N-1}(\partial B_1) +\frac{Q^2}{n} r_n^{-(N-\alpha)} \I_\alpha(B_1).
 \]
By the choice of $\beta$,
\[
\lim_{n\to +\infty} n r_n^{N-1}+n^{-1} r_n^{-(N-\alpha)}= \lim_{n\to +\infty} n^{-(\beta(N-1)-1)}+ n^{-(1-\beta (N-\alpha))}=0,
\]
which concludes the proof.
\end{proof}

Performing a more careful analysis it can be shown that   actually even local minimizers do not exist  for the Hausdorff topology (see \cite[Th. 3.4]{GNRI}). 
If the construction leading to the non-linear instability of the ball described here is made of many disconnected components, it has been proven in \cite[Th. 2]{MurNov} that (at least in the physical case $N=3$, $\alpha=2$) the ball is actually unstable
even in the class of smooth graphs over the ball.    
\begin{theorem}
 Let $N=3$ and $\alpha=2$. Then, for every $\delta>0$, there exists a smooth function $\phi_\delta \, :\, \partial B_1\to (-\delta,\delta)$ such that letting 
 \[
  \Omega_\delta:=\lt\{x \ : \ |x|\le 1+\phi_\delta\lt(\frac{x}{|x|}\rt)\rt\}
 \]
we have $|\Omega_\delta|=|B_1|$ and 
\[
 \cE_2(\Omega_\delta)<\cE_2(B_1).
\]
\end{theorem}

In the case $\alpha\in (0,1]$ one can expect a stronger interaction between both terms in \eqref{eq:mainprob} which might restore well-posedness. 
This has been recently investigated in \cite{MurNovRuf} in the case $N=2$, $\alpha=1$, which 
 corresponds to three dimensional drops trapped between two very close  isolating plates. The authors were able to completely solve this problem
 \begin{theorem}
  Let $N=2$ and $\alpha=1$. There exists an explicit $\overline{Q}$ such that for $Q\le \overline{Q}$, the only minimizer of \eqref{eq:mainprob} is given by the unit disk while for $Q>\overline{Q}$, there are no minimizers.
 \end{theorem}
\begin{proof}
A rough  idea of the proof is the following: the crucial observation is that in dimension two, the energy of a connected set decreases under convexification. Of course this operation does not preserves the volume. 
This motivates dropping the volume constraint and   studying  the global minimizer of the energy. By the above observation, a global  minimizer is made of a union of convex sets. 
Using the linearity  (respectively the sub-linearity) of the perimeter  (respectively of the Riesz capacity) with respect to the Minkowski sum,
it can be shown that the ball is the global minimizer of the energy amongst convex sets.  The global minimizer is thus made of a union of balls from which it is readily seen that it is actually a single ball.
Therefore, for any radius $R>0$ there exists a charge $Q(R)>0$ such that the ball of radius $R$ is the only minimizer (up to translation) of
\[
\min \mathcal E_1^{Q(R)}(\Omega).
\]
Here we adopted the notation $\mathcal E_1^{Q(R)}(\Omega)$ instead of $\mathcal E_1(\Omega)$ just to emphasize the dependence on the charge.  
Let $Q(1)$ be the charge associated to $B_1$. If $Q>Q(1)$, then using a construction similar to the one used in the proof Theorem \ref{theo:nonexist}, it is possible to prove non-existence of a minimizer while for $Q<Q(1)$,
since for every $\Omega$ with $|\Omega|=|B_1|$,
\[
\begin{aligned}
\mathcal E^Q_1 (\Omega) &= \mathcal E^{Q(1)}_{1} (\Omega) - (Q(1)^2 -
Q^2) \I_1(\Omega) 
\\
&\geq 
\mathcal E^{Q(1)}_{1}(  B_1) -
(Q(1)^2 - Q^2) \I_1(\Omega)  
\\
&\geq \mathcal E^{Q(1)}_{1}(  B_1) -
(Q(1)^2 - 
Q^2) \I_1( B_1) \\
&= \mathcal E^Q_1( B_1),
\end{aligned}
 \]
 with equality  if and only if $\Omega=B_1$, we obtain that the ball $B_1$ is the unique solution of \eqref{eq:mainprob}. Notice that in the second inequality we used  that the ball is a maximizer of $\mathcal I_1$ under volume constraint.
 \end{proof}

Turning back to the case $\alpha>1$ where \eqref{eq:mainprob} is ill-posed, it is natural to wonder if restricting the admissible set  
could restore well-posedness. A first possibility, explored in \cite{GNRI} is to add a strong constraint on the curvature. For $\delta>0$, we say that a set $\Omega$ satisfies the $\delta-$ball condition 
if for every $x\in \partial \Omega$ there are two balls of radius $\delta$ touching at $x$, one of which is contained in $\Omega$ and the other one which is contained in $\Omega^c$. Notice that this implies in particular that $\partial \Omega$ is $C^{1,1}$ with all the curvatures bounded by $\delta^{-1}$.  
We set
\[
\mathcal A_\delta:=\left\{\Omega\subset\R^N\,:\,|\Omega|=|B_1|,\,\,\text{$\Omega$ satisfies the $\delta-$ball condition}  \right\}.
\]
 Under this regularity assumption, it was proven in \cite[Th. 4.3]{GNRI}  that minimizers exist for small enough charges while non-existence for large charges holds in the case $\alpha>N-1$ \cite[Th. 4.5]{GNRI}.
 \begin{proposition}
  For every $N\ge 2$ and $\alpha \in (0,N)$, there exists $Q_0(N,\alpha)>0$ such that for every $\delta$ small enough and every $Q<Q_0 \delta^N$ a minimizer of
  \begin{equation}\label{prob:curvature}
   \min_{\Omega\in \mathcal{A}_\delta} \cE_\alpha(\Omega)
  \end{equation}
exists. Moreover, if $\alpha>N-1$, there exists $Q_1(N,\alpha)>0$ such that for every $\delta$ small enough and every $Q>Q_0 \delta^{-((N-\alpha)(N-1)+1)/2}$, no minimizer of \eqref{prob:curvature} exists. 
 \end{proposition}
\begin{proof}
 For the existence part the main point is to  prove that every minimizing sequence $\Omega_n$ must be connected for $Q<Q_0 \delta^N$. By (almost) minimality, we have 
 \begin{equation}\label{eq:basicquant}
  \H^{N-1}(\partial \Omega_n)-\H^{N-1}(\partial B_1)\le Q^2 (\I_\alpha(B_1)-\I_\alpha(\Omega_n)).
 \end{equation}
The quantitative isoperimetric inequality \cite{FuMaPra} then implies that $|\Omega_n\Delta B_1|\les Q$. Thanks to the $\delta-$ball condition, this yields that $\Omega_n$  is indeed connected for $Q\ll \delta^N$.\\
The non-existence part is obtained by constructing a competitor made of $\delta^{-N}$ balls of radius $\delta$.
\end{proof}
In the Coulombic case $\alpha=2$, it was shown in \cite[Th. 5.6]{GNRI} (see also \cite[Cor. 6.4]{GNRI} for the logarithmic case when $N=2$) that for small enough charges, the ball is the unique minimizer of \eqref{prob:curvature}.
\begin{theorem}\label{theorem1}
	Let $N\ge 2$ and  $\alpha=2$. Then there exists $Q_0(N,\delta)$ such that for $Q\le Q_0$,  $B_1$ is the only minimizer (up to translation) of problem \eqref{prob:curvature}.
	\end{theorem}
The proof of this result is quite long and involved but the basic idea is to argue as in \cite{CicLeo,KM2,figalli15, GolMer} for instance and  show that for small charges  minimizers
 are nearly spherical sets, that is small Lipschitz graphs over $\partial B_1$. This allows the use of   a Taylor
expansion of the perimeter for this type of sets given by Fuglede \cite{fuglede}. The main technical lemma is the following (see \cite[Prop. 5.5]{GNRI} and the proof of \cite[Th. 5.6]{GNRI}).
\begin{lemma}\label{lem:tech}
 For $N\ge 2$ and $\alpha=2$, if $\Omega$ is a nearly spherical set and if the optimal measure $\mu$ is bounded in $L^\infty(\partial \Omega)$, then there exists a constant $C$ depending on this $L^\infty$ bound such that 
 \begin{equation}\label{eq:quantcurvature}
  \I_2(B_1)-\I_2(\Omega)\le C (\H^{N-1}(\partial\Omega)-\H^{N-1}(\partial B_1)).
 \end{equation}
\end{lemma}
Thanks to the $\delta-$ball condition, it can be proven that for $Q$ small enough, minimizers of \eqref{prob:curvature} satisfy the hypothesis of Lemma \ref{lem:tech}. The proof of Theorem \ref{theorem1} is concluded by combining \eqref{eq:quantcurvature} together with \eqref{eq:basicquant}.

\begin{remark}
	One consequence of Theorem \ref{theorem1} is the stability of the ball under small $C^{1,1}$ perturbations. This extends a previous result of  \cite{FonFri} where stability with respect to $C^{2,\alpha}$ perturbations was proven. Let us however point out that in \cite{FonFri}, the asymptotic stability of the ball is also studied. 
\end{remark}

 An alternative way to restore well-posedness for \eqref{eq:mainprob} is to reduce the admissible class to convex sets. If this geometric restriction is not directly comparable with the $\delta-$ball condition, 
 it allows for less regular competitor (Lipschitz). 
As shown in \cite[Th. 2.3]{GNRII}, under the convexity constraint, there is always a minimizer for \eqref{eq:mainprob}.
\begin{theorem}
	For every $N\ge 2$, $\alpha\in(0,N]$ and $Q>0$, there exists a minimizer  of
	\begin{equation}\label{problemconvex}
	\min \left\{\cE_\alpha(\Omega)\,:\, \text{$|\Omega|=|B_1|$, $\Omega$ convex}  \right\}.
	\end{equation}
\end{theorem} 
Having Theorem \ref{theorem1} in mind, it is natural to wonder if, at least in the Coulombic case, \eqref{problemconvex} is still minimized by the unit ball for small $Q$. In the bi-dimensional logarithmic case, it has been proven to hold in \cite[Th. 5.1]{GNRII}. 
\begin{theorem}\label{theo:ballconv}
 Let $N=\alpha=2$ then for $Q$ small enough, the only minimizer of \eqref{problemconvex} is the unit ball.
\end{theorem}
The idea is to show that for small charges, minimizers of \eqref{problemconvex} satisfy the hypothesis of Lemma \ref{lem:tech}. This is a consequence of the following regularity result \cite[Th. 4.4]{GNRII} .
\begin{theorem}\label{theo:reg}
 For $N=\alpha=2$ and every $Q>0$, every minimizer of \eqref{problemconvex} is $C^{1,1}$, with uniform $C^{1,1}$ bounds for small $Q$.
\end{theorem}
This result proves that in dimension two and when restricted to the class of convex sets, conical singularities never appear. The proof of Theorem \ref{theo:reg} is quite long and technical 
but the main idea is to show that if $\Omega$ is not regular enough, then we can lower the energy by replacing part of the boundary by a straight line. The major difficulty is to precisely estimate the variation of the non-local term. 
A crucial technical point is that for convex sets, the optimal charge distribution is in $L^p(\partial \Omega)$ for some $p>2$ (see \cite[Th. 3.1]{GNRII}). In higher dimensions, it seems difficult to obtain regularity by such simple cutting-by-planes argument but it would be interesting to investigate further this question.   \\

Instead of imposing constraints on the admissible sets, another way of restoring well-posedness is to take into account regularizing mechanisms in the functional. 
One possibility, proposed in \cite{MurNov} is to take into consideration entropic effects and impose that the charge  is distributed in $\Omega$. The functional then becomes (in the physical case $N=3$, $\alpha=2$)
\[
 \mathcal{G}(\Omega):=\H^{N-1}(\partial \Omega)+Q^2 \mathcal{J}(\Omega),
\]
where 
\[
 \mathcal{J}(\Omega):=\min_{(v,\rho)} \lt\{ \int_{\R^3} |\nabla v|^2 +\int_{\Omega} \rho^2 \ : \ -\Delta v= \rho \textrm{  in } \R^3 \textrm{ and } \int_{\Omega} \rho=1 \rt\}.
\]
We refer to \cite{MurNov} for a physical motivation of this model. Their main result is existence of minimizers for this functional (see \cite[Th. 3]{MurNov}).
\begin{theorem}
 For every $Q>0$ and every $R>1$, there exists a minimizer of 
 \[
  \min \lt\{ \mathcal{G}(\Omega) \ : \ |\Omega|=|B_1| \textrm{ and } \Omega\subset B_R\rt\}.
 \]

\end{theorem}

Not much is known at the moment about the  regularity of the minimizers of this problem. 
The stability of the ball is also an open question.

\section{A perfectly conducting drop in a uniform external field}\label{section3}
The problem of finding the equilibrium shape of charged droplets is closely related to the problem of finding the equilibrium shape of a conducting drop submitted to an external electric field. 
If the understanding of meteorological phenomena  has first motivated the study of this question \cite{TayWil}, 
the wide spectrum of modern applications ranging from the breakdown of dielectrics due to the presence of water droplets to ink-jet printers might explain the large amount of literature on the subject (see for instance \cite{miksis,DubashMestel,karya}).
In the simplest and most considered case of a constant external field $\mathbb{E}$, we have (recall \eqref{probF}) for $\Omega\subset \R^N$ with $N\ge 3$,
\[
 \mathcal{F}(\Omega)=\min_{\mu(\Omega)=0} I_2(\mu) -\int_\Omega \mathbb{E}\cdot x d\mu
\]
and we look for a minimizer of 
\begin{equation}\label{prob:extfield}
 \min_{|\Omega|= |B_1| }\H^{N-1}(\partial \Omega)+\F(\Omega).
\end{equation}

It is quite easy to see that this problem is ill-posed.
\begin{theorem}
For every $N\ge 3$, and every $\mathbb{E}\in \R^N$,
\[
\inf_{|\Omega|=|B_1|} \H^{N-1}(\partial \Omega)+ \F(\Omega)=-\infty.
\]
\end{theorem}
\begin{proof}
Without loss of generality, we may assume that $\mathbb{E}=e_1:=(1,0,\dots,0)$.
Consider for $n\in\mathbb N$ the admissible set $\Omega_n$ made of two disjoint balls $B_n^{\pm}$ of measure $\frac{|B_1|}{2}$ and centers $ \pm ne_1$. 
We then have 
\[
\H^{N-1}(\partial \Omega_n)+ \F(\Omega_n)=2 \H^{N-1}(\partial B_n^{\pm})+\F(B_n^-\cup B_n^+).
\]
Using  $\mu_n:=\chi_{B_n^+}-\chi_{B_n^-}$ as test measure for $\F(B_n^-\cup B_n^+)$, we find
\begin{align*}
\F(\Omega_n)&\le I_2(\mu_n)-\int_{\Omega_n} x_1 d\mu_n\\
&= 2 I_2(\chi_{B_n^{\pm}})-2\int_{B_n^+\times B_n^-} \frac{dx dy}{|x-y|^{N-2}} -\int_{\Omega_n} (\chi_{B_n^+}-\chi_{B_n^-})x_1 dx\\
&\le 2 I_2(\chi_{B_n^{\pm}})- 2n|B_n^{\pm}|,
\end{align*} 
 which goes to $-\infty$ as $n\to+\infty$.
\end{proof}
Contrarily to the case of charged liquid drops, ill-posedness still holds in the class of convex sets.
\begin{theorem}
For every $N\ge 3$ and every $\mathbb{E}\in \R^N$,
\[
\inf_{\stackrel{|\Omega|=|B_1|}{\Omega \textrm{ convex}}} \H^{N-1}(\partial \Omega)+ \F(\Omega)=-\infty.
\]
\end{theorem} 
\begin{proof}
	As before, we may assume that $\mathbb{E}=e_1$. For  $n\in \N$, consider the set 
	\[
	\Omega_n:=\left\{x=(x_1,\dots,x_N)\in\R^N\,:\, |x_1|\le \frac{n}{2}\,\,\, |x_i|\le \frac{\varepsilon_n}{2},\,\, i=2,\dots,N \,\right\},
	\]
	where 
	\[	
	\varepsilon_n:=\left(\frac{|B_1|}{n}\right)^\frac{1}{N-1}.
	\]
	Notice that $\varepsilon_n$ is chosen so that $|\Omega_n|=|B_1|$. By definition of $\Omega_n$ we have
	\[
	\H^{N-1}(\partial \Omega_n)\les n\varepsilon^{N-2} \les n^{\frac{1}{N-1}}.
	\]
	Moreover, by letting 
	\[
	\Omega^-_n:=\Omega_n\cap \lt\{ x \ : \  x_1\in\lt[-\frac{n}{2},-\frac{n}{2}+\eps_n\rt]\rt\}, \qquad  \Omega^+_n:=\Omega_n\cap \lt\{ x \ : \  x_1\in\lt[\frac{n}{2}-\eps_n,\frac{n}{2}\rt]\rt\} 
	\]
	and then
	\[
	\mu_n:=\frac{\chi_{\Omega^+_n}}{|\Omega^+_n|}-\frac{\chi_{\Omega^-_n}}{|\Omega^-_n|},
	\] 
	we have that $\mu_n$ is admissible for $\F(\Omega_n)$ and thus $\F(\Omega_n)\le F(\mu_n)$.  Since on the one hand,
	\[
	\int_{\Omega_n} x_1\,d\mu_n\ges n 
	\]
	and on the other hand 
	\[
	I(\mu_n)\les \frac{1}{|\Omega^{+}_n|^2}\int_{\Omega_n^+\times \Omega_n^+}\frac{dx dy}{|x-y|^{N-2}}\les \eps_n^{-(N-2)}\les n^\frac{N-2}{N-1},
	\]
	we find that 
	\[
	\H^{N-1}(\partial \Omega_n)+ \F(\Omega_n)\les n^{\frac{1}{N-1}}+n^\frac{N-2}{N-1}-n
	\]
	which diverges to $-\infty$ as $n\to+\infty$.
\end{proof}

\section{Open problems}\label{section4}
In this last section we state few open problems. 
\begin{itemize}
	\item[i)] As already pointed out, in light of the large physical literature about the Taylor cones, it would be interesting to find a reasonable setting where these can be rigorously studied.
	\item [ii)] A first step would be for instance to understand if Theorem \ref{theo:ballconv} still holds in dimension $N\ge 3$ or if conical singularities can appear in the class of convex sets.
	\item[iii)] It is still an open question to know if the ball is stable under small  Lipschitz deformations for small charges. 
	\item[iv)] It is natural to try to extend the stability analysis for the ball to $\alpha\neq 2$.
	\item[v)] In light of Theorem \ref{theorem1}, it would be interesting to see if well-posedness holds (for small charge) when $\alpha\in[N-1,N)$.
	\item[vi)] Not much is known about the minimizers of the functional $\mathcal{G}$ introduced in \cite{MurNov}. It would be interesting to study existence/non-existence of minimizers without confinement, their regularity and the stability of the ball. 
	\item[vii)] Another  natural way to regularize \eqref{eq:mainprob} would be to add a curvature term of Willmore type in the energy. One would then study the functional
	\[
	 \H^{N-1}(\partial \Omega)+ \I_\alpha(\Omega)+\int_{\partial \Omega} \kappa^2,
	\]
where $\kappa$ is the mean curvature of $\partial \Omega$.
	\item[viii)] The physical model behind the electrowetting technique is very similar to the ones studied here. This technique, which is used for optical devices
	and electronic displays consists in applying a voltage on a sessile conducting drop. As already observed by Lippmann \cite{Lip} in 1875, this leads to a modification of the (apparent) contact angle, while the microscopic contact angle is still the one given by the classical Young law \cite{Scheid}. 
	It was discovered later on that the macroscopic angle decreases until reaching a saturation angle (see the review paper \cite{mugele}). 
	Despite its importance for application, there has been only few rigorous results about electrowetting (see \cite{fontkindI,fontkindII}) and it would be interesting to study both the formation of the macroscopic angle and the saturation phenomenon.  
\end{itemize}
\frenchspacing
\bibliographystyle{alpha}
\bibliography{dropinfield}

\end{document}